\documentclass[10pt]{amsart}
\usepackage{amssymb, amscd, amsmath, amsthm, epsf, epsfig, latexsym, enumerate}

\renewcommand{\geq}{\geqslant}
\renewcommand{\leq}{\leqslant}

\newtheorem{theorem}{Theorem}
\newtheorem{lemma}[theorem]{Lemma}
\newtheorem{cor}[theorem]{Corollary}

\newtheorem*{thm}{Theorem}

\newtheorem*{cor*}{Corollary}

\begin{document}

\title{$PD_3$-groups and HNN extensions}

\author{Jonathan A. Hillman }
\address{School of Mathematics and Statistics\\
     University of Sydney, NSW 2006\\
      Australia }

\email{jonathan.hillman@sydney.edu.au}

\begin{abstract}
We show that if a $PD_3$-group $G$ splits as an HNN extension $A*_C\varphi$
where $C$ is a $PD_2$-group then the Poincar\'e dual in 
$H^1(G;\mathbb{Z})= Hom(G,\mathbb{Z})$ of the homology class $[C]$
is the epimorphism $f:G\to\mathbb{Z}$ with kernel the normal closure of $A$.
We also make several other observations about $PD_3$-groups 
which split over $PD_2$-groups.
\end{abstract}



\maketitle

In this note we shall give algebraic analogues of some 
properties of Haken 3-manifolds.
We are interested in the question ``when does a $PD_3$-group
split over a $PD_2$-group?".
In \S2 we show that such splittings are minimal in a natural partial order 
on splittings over more general subgroups.
In the next two sections we consider $PD_3$-groups $G$ which
split as an HNN extension
$A*_C\varphi$ with $A$ and $C$ finitely generated.
In \S3 we show that $A$ and $C$ have the same number of indecomposable factors.
Our main result is in \S4, 
where we show that if $C$ is a $PD_2$-group then the Poincar\'e dual in 
$H^1(G;\mathbb{Z})= Hom(G,\mathbb{Z})$ of the homology class $[C]$
is the epimorphism $f:G\to\mathbb{Z}$ with kernel the normal closure of $A$.
In \S5 we extend an argument from \cite{Hi03} to show that no $FP_2$ subgroup 
of a $PD_3$-group is a properly ascending HNN extension,
and in \S6 we show that if $G$ is residually finite and splits over a $PD_2$-group then $G$ has a subgroup of finite index with infinite abelianization.
Our arguments extend readily to $PD_n$-groups with $PD_{n-1}$-subgroups, 
but as our primary interest is in the case $n=3$,
we shall formulate our results in such terms.

\section{terminology}

We mention here three properties of 3-manifold groups that are not 
yet known for all $PD_3$-groups: coherence, residual finiteness and having subgroups of finite index with infinite abelianization.
Coherence may often be sidestepped by requiring 
the subgroups in play to be $FP_2$ rather than finitely generated. 
If every finitely generated subgroup of a group $G$ is $FP_2$ 
we say that $G$  is {\it almost coherent}.

We shall say that a group $G$ is {\it split\/} over a subgroup $C$ 
if it is either a generalized free product with amalgamation (GFPA) 
$G=A*_CB$, where $C<A$ and $C<B$,
or an HNN extension $G=HNN(A;\alpha,\gamma:C\to{A})$,
where $\alpha$ and $\gamma$ are monomorphisms.
(We may also write $G=A*_C\varphi$, where $\varphi=\gamma\circ\alpha^{-1}$.)
An HNN extension is {\it ascending\/} if one of the associated subgroups
is the base. 
In that case we may assume that $\alpha=id_A$, and $\varphi=\gamma$
is an injective endomorphism of $A$.

The {\it virtual first Betti number} $v\beta(G)$ of a finitely generated group 
is the least upper bound of the first Betti numbers $\beta_1(N)$ of normal subgroups $N$ of finite index in $G$.
Thus $v\beta(G)>0$ if some subgroup of finite index maps onto $\mathbb{Z}$.

A group $G$ is {\it large\/} if it has a subgroup of finite index 
which maps onto a non-abelian free group.
It is clear that if $G$ is large then $v\beta(G)=\infty$.

\section{comparison of splittings}

Let $G$ be a group which is a GFPA $A*_CB$ or an HNN extension $A*_C\varphi$.
If we identify the groups $A, B$ and $C$ with subgroups of $G$ then 
inclusion defines a partial order on such splittings:
$A*_CB\leq{A'*_{C'}B'}$ if $A\leq{A'},B\leq{B'}$ and $C\leq{C'}$,
and $A*_C\varphi\leq{A'*_{C'}\varphi'}$ if $A\leq{A'},C\leq{C'}$
and $\varphi'|_C=\varphi$, and the stable letters coincide.
(In the HNN case we are really comparing splittings compatible 
with a given epimorphism
$G\to\mathbb{Z}\cong{G/\langle\langle{A}\rangle\rangle}$.)

\begin{lemma}
\label{brit}
Let $G=A'*_C\varphi$ be an HNN extension,
with stable letter $t$,
and let $A\leq{A'}$ be a subgroup such that $C\cup\varphi(C)\leq{A}$.
If $G=\langle{A,t}\rangle$ then $A=A'$.
\end{lemma}

\begin{proof}
Let $\alpha\in{A'}$. 
Then we may write $\alpha=a_0t^{\varepsilon_1}a_1\dots{t^{\varepsilon_n}a_n}$
where $a_i\in{A}$ and $\varepsilon_i=\pm1$, for all $i$,
since $G=\langle{A,t}\rangle $.
We may clearly assume that $n$ is minimal.
Hence there are no substrings of the form $tct^{-1}$
or $t^{-1}\varphi(c)t$, with $c\in{C}$, in this expression for $\alpha$
(since any such may be replaced by $\varphi(c)$ or $c$, respectively).
But it then follows from Britton's Lemma for the HNN extension $A'*_C\varphi$
that $n=0$, and so $\alpha=a_0$ is in $A$.
\end{proof}

If $G$ is a $PD_3$-group then we would like to know when 
$C$ can be chosen to be a $PD_2$-group. 

\begin{lemma}
\label{pd2 minimal}
Let $G$ be a $PD_3$-group which is a generalized free product 
with amalgamation $A*_CB$ or an HNN extension $A*_C\varphi$, 
with $C$ a $PD_2$-group. 
Then the splitting is minimal in the partial order determined by inclusions.
\end{lemma}

\begin{proof}
Suppose that $A'*_{C'}B'\leq{A*_CB}$ or $A'*_{C'}\varphi'\leq{A*_C\varphi}$
(respectively), is another splitting for $G$.
Then $C'$ is either a free group or has finite index in $C$.
The inclusions induce a commuting diagram relating the Mayer-Vietoris sequences associated to the splittings.
In each case, the left hand end of the diagram is
\begin{equation*}
\begin{CD}
0\to{H_3(G;\mathbb{Z})}@>\delta'>>{H_2(C';\mathbb{Z})}\\
@V=VV @VVV\\
0\to{H_3(G;\mathbb{Z})}@>\delta>>{H_2(C;\mathbb{Z})}.
\end{CD}
\end{equation*}
Since the connecting homomorphisms $\delta'$ is injective, 
$H_2(C';\mathbb{Z})\not=0$, and so $C'$ cannot be a free group.
Hence it is a $PD_2$-group, 
and so $\delta$ and $\delta'$ are isomorphisms \cite{BE78}.
Since the inclusion of $C'$ into $C$ has degree 1, we see that $C'=C$.
If $G=A'*_C\varphi$ it then follows from Lemma \ref{brit} that $A'=A$.
If $G=A*_CB$ and $G=A'*_CB'$ then a similar argument based on 
normal forms shows that $A'=A$ and $B'=B$. 
\end{proof}

If $f:G\to\mathbb{Z}$ is an epimorphism then $G\cong{A*_C\varphi}$ 
with $\mathrm{Ker}(f)=\langle\langle{A}\rangle\rangle$ and 
stable letter represented by $t\in{G}$ with $f(t)=1$.
For instance, we may take $C=A=\mathrm{Ker}(f)$ and $\varphi$ to be
conjugation by $t$.
If $\mathrm{Ker}(f)$ is finitely generated, 
this is the only possibility 
(up to the choice of $t$ with $f(t)=\pm1$), 
but in general there are other ways to do this.
If $G$ is $FP_2$ then we may choose $A$ and $C$ finitely generated \cite{BS78}, 
and if $G$ is almost coherent then $A$ and $C$ are also $FP_2$.
The construction of \cite{BS78} gives a pair $(A,C)$ with $A$ 
generated by $C\cup\varphi(C)$, 
which is usually far from minimal in this partial order.
(See below for an example.)
If $G$ is $FP$ then $A$ is $FP_k$ if and only if $C$ is $FP_k$,
for any $k\geq1$ \cite[Proposition 2.13]{Bi}.

If $G$ is $FP_2$ and $\mathrm{Ker}(f)$ is not finitely generated 
then any HNN structure for $G$ with finitely generated base and 
associated subgroups is the initial term of an infinite increasing chain 
of such structures, obtained by applying the construction of \cite{BS78}.
If $G=A*_A\varphi$ is a properly ascending HNN extension,
so that $\varphi(A)<A$, 
then $G$ has a doubly infinite chain of HNN structures,
with bases the subgroups $t^nAt^{-n}$, for $n\in\mathbb{Z}$.
However $PD_n$-groups are never properly ascending HNN extensions.
(See Theorem \ref{no proper ascension} below.)
Does every descending chain of HNN structures 
for a $PD_3$-group terminate?
Do any $PD_3$-groups which are HNN extensions have minimal splittings over
$FP_2$-groups which are not $PD_2$-groups?

Let $T_2$ be the orientable surface of genus 2.
The $PD_2$-group $H=\pi_1(T_2)$ has a standard presentation
\[
\langle{a,b,c,d}\mid {[a,b][c,d]=1}\rangle.
\]
We may rewrite this presentation as
\[
\langle{a,b,c,t}\mid {tct^{-1}=aba^{-1}b^{-1}c}\rangle,
\]
which displays $H$ as an HNN extension 
$F(a,b,c)*_{\langle{c}\rangle}\varphi$, 
split over the $PD_1$-group $\langle{c}\rangle\cong\mathbb{Z}$.
The associated epimorphism $f:H\to\mathbb{Z}$ is determined by
$f(a)=f(b)=f(c)=0$ and $f(d)=1$.
In this case the algorithm from \cite{BS78} would suggest taking
$C=\langle{a,b,c}\rangle$ and
$A=\langle{a,b,c,tat^{-1},tbt^{-1}}\rangle$,
giving an HNN extension with base $A\cong{F(5)}$ and split over $C\cong{F(3)}$.
Taking products, we see then that the $PD_3$-group 
$G=\pi_1(T_2\times{S^1})=H\times\mathbb{Z}$ splits over the $PD_2$-group $\mathbb{Z}^2$,
and is also an HNN extension with
base $F(5)\times\mathbb{Z}$ and associated subgroups $F(3)\times\mathbb{Z}$.
The latter groups have one end, but are not $PD_2$-groups.

\section{indecomposable factors}

If $G$ is a $PD_3$-group then $c.d.A=c.d.C=2$, 
since these subgroups have infinite index in $G$, 
and $H_2(C;\mathbb{Z})\not=0$,
as observed in Lemma \ref{pd2 minimal}.
A simple Mayer-Vietoris argument shows that $H^1(A;\mathbb{Z}[G])\cong{H^1(C;\mathbb{Z}[G])}$ as right $\mathbb{Z}[G]$-modules, since $H^i(G;\mathbb{Z}[G])=0$ for $i\leq2$.
(Note that the latter condition fails for $PD_2$-groups.)
The isomorphism is given by the difference $\alpha_*-\gamma_*$
of the homomorphisms induced by $\alpha$ and $\gamma$.

We shall assume henceforth that $A$ and $C$ are finitely generated.
Then these modules may be obtained by extension of coefficients 
from the ``end modules" $H^1(A;\mathbb{Z}[A])$ and $H^1(C;\mathbb{Z}[C])$. 
If one is 0 so is the other,
and so $A$ has one end if and only if $C$ has one end.
If $A$ and $C$ are $FP_2$ and have one end then they are 2-dimensional duality groups, and we may hope to apply the ideas of \cite{KK05}.

Can $G$ have splittings with base and associated subgroups 
having more than one end?
The next lemma implies that the subgroups $A$ and $C$ must have 
the same numbers of indecomposable factors.
(The analogous statement for $PD_2$-groups is false,
as may be seen from the example in \S2 above!)

\begin{lemma}
\label{freeE^1}
Let $K=(*_{i=1}^mK_i)*F(n)$ be the free product of $m\geq1$ 
finitely generated groups $K_i$ with one end and 
$n\geq0$ copies of $\mathbb{Z}$.
Then $H^1(K;\mathbb{Z}[K])\cong\mathbb{Z}[K]^{r-1}$,
where $r=m+n$ is the number of indecomposable factors of $K$.
\end{lemma}

\begin{proof}
If $n=0$ the result follows from the Mayer-Vietoris sequence for the free product, with coefficients $\mathbb{Z}[K]$.

In general, let $J=*_{i=1}^mK_i$ and let $C_*(J)$ be a resolution 
of the augmentation module $\mathbb{Z}$
by free $\mathbb{Z}[J]$-modules,
with $C_0(J)=\mathbb{Z}[J]$.
Then there is a corresponding resolution $C_*(K)$ of $\mathbb{Z}$
with $C_q(K)\cong\mathbb{Z}[K]\otimes_{\mathbb{Z}[J]}C_q(J)$
if $q\not=1$ and $C_1(K)\cong
\mathbb{Z}[K]\otimes_{\mathbb{Z}[J]}C_q(J)\oplus\mathbb{Z}[K]^n$.
Hence there is a short exact sequence of chain complexes 
(of left $\mathbb{Z}[K]$-modules)
\[
0\to\mathbb{Z}[K]\otimes_{\mathbb{Z}[J]}C_*(J)\to{C_*(K)}\to
\mathbb{Z}[K]^n\to0,
\]
where the third term is concentrated in degree 1.
The exact sequence of cohomology with coefficients $\mathbb{Z}[K]$
gives a short exact sequence of {\it right} $\mathbb{Z}[K]$-modules
\[
0\to\mathbb{Z}[K]^n\to{H^1(K;\mathbb{Z}[K])}\to
{H^1(Hom_{\mathbb{Z}[K]}
(\mathbb{Z}[K]\otimes_{\mathbb{Z}[J]}C_*(J),\mathbb{Z}[K])}\to0.
\]
We may identify the right-hand term with
$H^1(J;\mathbb{Z}[J])\otimes_{\mathbb{Z}[J]}\mathbb{Z}[K]
\cong\mathbb{Z}[K]^{m-1}$, since $J$ is finitely generated.
The lemma follows easily.
\end{proof}

The lemma applies to $A$ and $C$,
since they are finitely generated and torsion-free.
The indecomposable factors of $C$ are either conjugate to subgroups of
indecomposable factors of $A$ or are infinite cyclic,
by the Kurosh subgroup theorem.
If $A$ and $C$ have no free factors and the factors of $C$ are conjugate 
into distinct factors of $A$ then, 
after modifying $\varphi$ appropriately,
we may assume that $\alpha(C_i)\leq{A_i}$, for all $i$.
However, we cannot expect to also normalize $\gamma$ in a similar fashion.

\section{the dual class}

If $M$ is a closed 3-manifold with $\beta_1(M)>0$ then there is 
an essential map $f:M\to{S^1}$.
Transversality and the Loop Theorem together imply that
there is a closed incompressible surface $S\subset{M}$ such that 
$M\setminus{S}$ is connected. 
Hence $\pi_1(M)$ is an HNN extension with base $\pi_1(M\setminus{S})$ 
and associated subgroups copies of $\pi_1(S)$.
Moreover, the stable letter of the extension is represented by a simple closed curve in $M$ which intersects $S$ transversely in one point.
Let $w=w_1(M)$.
Then $w_1(S)=w|_S$ and the image of the fundamental class $[S]$ 
in $H_2(M;\mathbb{Z}^w)$ is Poincar\'e dual to the image of $f$ in $H^1(M;\mathbb{Z})=[M,S^1]$.

There is no obvious analogue of transversality in group theory.
Nevertheless a similar result holds for $PD_3$-groups.
(We consider only the orientable case, for simplicity.)

\begin{theorem}
\label{dual class}
Let $G=HNN(A;\alpha,\gamma:C\to{A})$ be an orientable $PD_3$-group 
which is an HNN extension split over a $PD_2$-group $C$.
Let $f\in{H^1(G;\mathbb{Z})}$ be the epimorphism 
with kernel $\langle\langle{A}\rangle\rangle_G$.
Then $f\frown[G]$ is the image of $[C]$ in $H_2(G;\mathbb{Z})$, 
up to sign.
\end{theorem}

\begin{proof}
The subgroup $C$ is orientable and the pair $(A;\alpha,\gamma)$ 
is a $PD_3^+$-pair \cite[Theorem 8.1]{BE78},
and so there is an exact sequence
\begin{equation*}
\begin{CD}
H_3(A,\partial;\mathbb{Z})@>(1,1)>>H_2(C;\mathbb{Z})\oplus{H_2(C;\mathbb{Z})}
@>(\alpha_*,-\gamma_*)>>{H_2(A;\mathbb{Z})}\to
{H_2(A;\partial;\mathbb{Z})}.
\end{CD}
\end{equation*}
Hence $\alpha_*[C]=\gamma_*[C]$, and the subgroup they generate is 
an infinite cyclic direct summand of $H_2(A;\mathbb{Z})$, 
since $H_2(A;\partial;\mathbb{Z})\cong{H^1(A;\mathbb{Z})}$ is free abelian.

Let $t\in{G}$ correspond to the stable letter for the $HNN$ extension, 
and let 
$A_j=t^jAt^{-j}$, $\alpha_j(c)=t^j\alpha(c)t^{-j}$ and 
$\gamma_j(c)=t^j\gamma(c)t^{-j}$,
for all $c\in{C}$ and $j\in\mathbb{Z}$.
Let $K_p$ be the subgroup generated by $\cup_{|j|\leq|p|}A_j$, for $p\geq0$.
Then $K_0=A$ and 
\[
K_{p+1}=A_{-p-1}*_{\alpha_{-p}=\gamma_{-p-1}}
K_p*_{\alpha_{p+1}=\gamma_p}A_{p+1},\quad{for~all~~p\geq0},
\]
and $K=\langle\langle{A}\rangle\rangle_G=\mathrm{Ker}(f)$
is the increasing union $K=\cup{K_p}$  
of iterated amalgamations with copies of $A$ over copies of $C$.
Each pair $(K_p;\alpha_{-p},\gamma_p)$ is again a $PD_3^+$-pair,
and so the images of $[C]$ in $H_2(K;\mathbb{Z})$ under the 
homomorphisms induced by the $\alpha_n$s all agree.

Let $\Lambda=\mathbb{Z}[G/K]=\mathbb{Z}[t,t^{-1}]$,
and let $\varepsilon:\Lambda\to\mathbb{Z}$ be the augmentation.
If $M$ is a right $\Lambda$-module let $\overline{M}$ denote the left module with the conjugate action given by $t.m=mt^{-1}$, for all $m\in{M}$.

The homology groups $H_i(K;\mathbb{Z})=H_2(G;\Lambda)$ are finitely generated left 
$\Lambda$-modules,
with action deriving from the action of $G$ on $K$ by conjugation.
(The cohomology modules $H^j(G;\Lambda)$ are naturally right modules.)
Then $H_2(G;\Lambda)=H_2(K;\mathbb{Z})=\lim{H_2(K_p;\mathbb{Z})}$.
Since $t.\alpha_{n*}[C]=\alpha_{(n+1)*}[C]=\alpha_{n*}[C]$, for all $n$,
the image of $[C]$ in $H_2(K;\mathbb{Z})$ 
generates an infinite cyclic direct summand.

Poincar\'e duality gives a commutative diagram
\begin{equation*}
\begin{CD}
\overline{H^1(\mathbb{Z};\Lambda)}@>H^1(f)>>\overline{H^1(G;\Lambda)}@>\frown[G]>>
H_2(G;\Lambda)\\
@V\cong{}V\varepsilon_\#V @VV\varepsilon_\#V @VV\varepsilon_\#V\\
H^1(\mathbb{Z};\mathbb{Z})@>id_\mathbb{Z}\mapsto{f}>>
H^1(G;\mathbb{Z})@>\frown[G]>>H_2(G;\mathbb{Z})
\end{CD}
\end{equation*}
in which the vertical homomorphisms are induced by the change of coefficients 
$\varepsilon$ and the two right hand horizontal homomorphisms 
are Poincar\'e duality isomorphisms.
The module $H^1(G;\Lambda)$ an extension of $\overline{Hom_\Lambda(K/K',\Lambda)}$
by $Ext_\Lambda^1(\mathbb{Z},\Lambda)$,
by the Universal Coefficient spectral sequence.
Since $Hom_\Lambda(K/K',\Lambda)$ has no non-trivial $\Lambda$-torsion,
while $Ext_\Lambda^1(\mathbb{Z},\Lambda)\cong\Lambda/(t-1)\Lambda=\mathbb{Z}$,
the homomorphism $H^1(f)$ carries $H^1(\mathbb{Z};\Lambda)=
Ext_\Lambda^1(\mathbb{Z},\Lambda)\cong\mathbb{Z}$ onto the $\Lambda$-torsion submodule of $H^1(G;\Lambda)$.
A diagram chase shows that $f\frown[G]$ is the image of $[C]$
in $H_2(G;\mathbb{Z})$, up to sign.
\end{proof}

In \cite{KR88} it is shown that if a $PD_3$-group $G$ has  
a subgroup $S$ which is a $PD_2$-group then $G$ splits over a subgroup
commensurate with $S$ if and only if an invariant $sing(S)\in\mathbb{Z}/2\mathbb{Z}$ is 0,
and then $S$ is maximal among compatibly oriented commensurate subgroups.
Theorem \ref{dual class} suggests a slight refinement 
of this splitting criterion.

\begin{thm}
[Kropholler-Roller \cite{KR88}]
Let $G$ be an orientable $PD_3$-group and $S<G$ a subgroup 
which is an orientable $PD_2$-group.
Then
\begin{enumerate}
\item$G\cong{A*_TB}$ for some $T$ commensurate with $S$ $\Leftrightarrow$
$sing(S)=0$ and $[S]=0$ in $H_2(G;\mathbb{Z})$;
\item$G\cong{A*_T\varphi}$ for some $T$ commensurate with $S$ $\Leftrightarrow$
$sing(S)=0$ and $[S]$ has infinite order in $H_2(G;\mathbb{Z})$;
\item$G\cong{A*_S\varphi}$ $\Leftrightarrow$
$sing(S)=0$ and $[S]$ generates an infinite direct summand of $H_2(G;\mathbb{Z})$.
\end{enumerate}
\end{thm}

\begin{proof}
The group $G$ splits over a subgroup $T$ commensurate with $S$ if and only if $sing(S)=0$ \cite{KR88},
and $[S]$ and $[T]$ are then proportional.
If $G=A*_TB$ is a generalized free product with amalgamation over 
a $PD_2$-group $T$ then the pairs $(A,T)$ and $(B,T)$ 
are again $PD_3^+$-pairs \cite{BE78}.
The image of $[T]$ in $H_2(G;\mathbb{Z})$ is trivial, since $T$
bounds each of $(A,T)$ and $(B,T)$, and so $[S]=0$ also.

If $G\cong{A*_T\varphi}$ is an HNN extension then the Poincar\'e dual 
of $[T]$ is an epimorphism $f:G\to\mathbb{Z}$, by the theorem, 
and so $[T]$ generates an infinite cyclic direct summand 
of $H_2(G;\mathbb{Z})$. 
Hence $[S]$ also has infinite order.
\end{proof}

If $[C]=[S]$ and $sing(S)=0$ is $sing(C)=0$ also?

Splittings over $PD_2$-groups dual to a given epimorphism need not be unique.
Let $W$ be an aspherical orientable 3-manifold with incompressible
boundary and two boundary components $U,V$.
Let $M=DW$ be the double of $W$ along its boundary.
Then $M$ splits over copies of $U$ and $V$, 
and $[U]=[V]$ in $H_2(M;\mathbb{Z})$.
If $U$ and $V$ are not homeomorphic the corresponding (minimal)
splittings of $G=\pi_1(M)$ are evidently distinct.
For instance, we may start with the hyperbolic 3-manifold of 
\cite[Example 3.3.12]{Th},
which is the exterior of a knotted $\theta$-curve  $\Theta\subset{S^3}$.
Let $W$ be obtained by deleting an open regular neighbourhood of a meridian
of one of the arcs of $\Theta$.
Then $W$ is aspherical, $\partial{W}=T\amalg{T_2}$ and each component of $\partial{W}$ is incompressible in $W$.

\section{no properly ascending HNN extensions}

Cohomological arguments imply that no $PD_3$-group is 
a properly ascending HNN extension \cite[Theorem 3]{Hi03}.
A stronger result holds for 3-manifold groups:
no finitely generated subgroup can be conjugate to a proper subgroup of itself \cite{Bu07}.
We shall adapt the argument of \cite{Hi03} to prove the corresponding result for $FP_2$ subgroups of $PD_3$-groups.

\begin{theorem}
\label{no proper ascension}
Let $H$ be an $FP_2$ subgroup of a $PD_3$-group $G$.
If $gHg^{-1}\leq{H}$ for some $g\in{G}$ then $gHg^{-1}=H$.
\end{theorem}

\begin{proof}
Suppose that $gHg^{-1}<H$. Then $g\not\in{H}$.
Let $\theta(h)=ghg^{-1}$, for all $h\in{H}$,
and let $K=H*_H\theta$ be the associated HNN extension, 
with stable letter $t$.
The normal closure of $H$ in $K$ is the union 
$\cup_{r\in\mathbb{Z}}t^rHt^{-r}$, and so
every element of $K$ has a normal form $k=t^mt^rht^{-r}$, 
where $m$ is uniquely determined by $k$, 
and $h$ is determined by $k,m$ and $r$.
Let $f:K\to{G}$ be the homomorphism defined by
$f(h)=h$ for all $h\in{H}$ and $f(t)=g$.
If $f(t^mt^rht^{-r})=f(t^nt^sh't^{-s})$ for some $m,n,r,s$
then $g^{n-m}=g^sh'g^{-s}g^th^{-1}g^{-t}$.
After conjugating by a power of $g$ if necessary,
we may assume that $s,t\geq0$, and so $g^{n-m}\in{H}$.
But then $H=g^{|n-m|}Hg^{-|n-m|}$.
Since $gHg^{-1}$ is a proper subgroup of $H$, we must have $n=m$.
It follows easily that $f$ is an isomorphism from $K$ to 
the subgroup of $G$ generated by $g$ and $H$.

Since $K$ is an ascending HNN extension with $FP_2$-base,
$H^1(K;\mathbb{Z}[K])$ is a quotient of $H^0(H;\mathbb{Z}[K])=0$
\cite[Theorem 0.1]{BG85}.
Hence it has one end.
Since no $PD_3$-group is an ascending HNN extension \cite[Theorem 3]{Hi03},
$K$ is a 2-dimensional duality group.
Hence it is the ambient group of a $PD_3$-pair $(K,\mathcal{S})$
\cite{KK05}.
Doubling this pair along its boundary gives a $PD_3$-group.
But this is again a properly ascending HNN extension,
and so cannot happen.
Therefore the original supposition was false, and so $gHg^{-1}=H$.
\end{proof}

\section{residual finiteness, splitting and largeness}

The fundamental group of an aspherical closed 3-manifold 
is either solvable or large \cite[Flowcharts 1 and 4]{AFW}.
This is also so for residually finite $PD_3$-groups containing 
$\mathbb{Z}^2$ \cite[Theorem 11.19]{Hi20}.
Here we shall give a weaker result for $PD_3$-groups which
split over other $PD_2$-groups.

\begin{theorem}
Let $G$ be a residually finite orientable $PD_3$-group which splits over 
an orientable $PD_2$-group $C$. Then either $\beta_1(G)>0$,
or $G$ maps onto $D_\infty$, or $G$ is large.
Hence $v\beta(G)>0$.
If $G$ is LERF and $\chi(C)<0$ then $G$ is large.
\end{theorem}

\begin{proof}
For the first assertion, 
we may assume that $\beta_1(G)=0$, and that $G\cong{A*_CB}$.
Then $(A,C)$ and $(B,C)$ are $PD_3$-pairs, 
and so $\beta_1(C)\leq2\beta_1(A)$ and $\beta_1(C)\leq2\beta_1(B)$.
Since $\beta_1(C)>0$, we must have $\beta_1(A)>0$ and $\beta_1(B)>0$ also.
Moreover $\beta_1(C)=\beta_1(A)+\beta_1(B)$, 
since $H_1(G)$ is finite and $H_2(G)=0$.
Hence $\beta_1(C)>\beta_1(A)$ and $\beta_1(C)>\beta_1(B)$.

Let $\{\Delta_n|n\geq1\}$ be a descending filtration of $G$ 
by normal subgroups of finite index.
Then $A_n=A/A\cap\Delta_n$, $B_n=B/B\cap\Delta_n$ and $C_n=C/C\cap\Delta_n$
are finite, and $G$ maps onto $A_n*_{C_n}B_n$, for all $n$.
If $A_n*_{C_n}B_n$ is finite then $C_n=A_n $ or $B_n$.
Thus if all these quotients of $G$ are finite we may assume that
$C_n=A_n$ for all $n$.
But then the inclusion of $C$ into $A$ induces an isomorphism on profinite completions, and so $\beta_1(C)=\beta_1(A)$, contrary to 
what was shown in the paragraph above.

If $C_n$ is a proper subgroup of both $A_n$ and $B_n$ then either
${[A_n:C_n]=[B_n:C_n]}=2$,
in which case $G$ maps onto $D_\infty$,
or one of these indices is greater than 2,
in which case $A_n*_{C_n}B_n$ is virtually free of rank $>1$,
and so $G$ is large.
In each case, it is clear that $v\beta(G)\geq1$.

Suppose now that $G$ is LERF.
If $[A_n:C_n]\leq2$ then $C_n$ is normal in $A_n$,
and so $C(A\cap\Delta_n)$ is normal in $A$.
Hence if $[A_n:C_n]\leq2$ for all $n$ then $\cap_nC(A\cap\Delta_n)$
is normal in $A$.
Since $G$ is LERF, this intersection is $C$.
Hence if both $[A_n:C_n]\leq2$ and $[B_n:C_n]\leq2$ for all $n$ then $C$ is normal in $G$, so $G$ is virtually a semidirect product $C\rtimes\mathbb{Z}$,
and is a 3-manifold group.
If $\chi(C)<0$ then $G$ is large \cite[Flowcharts 1 and 4]{AFW}.
\end{proof}

Remark. The lower central series of $D_\infty=Z/2Z*Z/2Z$ gives 
a descending filtration by normal subgroups of finite index 
which meets each of the free factors trivially.

Is every $PD_3$-group either solvable or large?



\begin{thebibliography}{99}

\bibitem{AFW} Aschenbrenner, M., Friedl, S. and Wilton, H.
\textit{3-Manifold Groups},

EMS Series of Lectures in Mathematics,

European Mathematical Society, Z\"urich (2015).

\bibitem{Bi} Bieri, R. \textit{Homological Dimensions of Discrete Groups}, 

Queen Mary College Mathematical Notes, London (1976).

\bibitem{BE78} Bieri, R. and Eckmann, B. Relative homology and Poincar\'e
duality for group pairs, 

J. Pure Appl. Alg. 13 (1978), 277--319.

\bibitem{BS78} Bieri, R. and Strebel, R. Almost finitely presentable 
soluble groups, 

Comment. Math. Helvetici 53 (1978), 258--278.

\bibitem {BG85} Brown, K. S. and Geoghegan, R. Cohomology with free
coefficients of the fundamental group of a graph of groups, 
Comment. Math. Helvetici 60 (1985), 31--45.

\bibitem{Bu07} Button, J. O. Mapping tori with first Betti number at least two,

J. Math. Soc. Japan 59 (2007), 351--370.





\bibitem{Hi03} Hillman, J. A. Tits alternatives and low dimensional topology,

J. Math. Soc. Japan 55 (2003), 365--383. 

\bibitem{Hi20} Hillman, J. A. \textit{Poincar\'e Duality in Dimension 3},

The Open Book Series, MSP, Berkeley (2020), to appear.


\bibitem{KK05} Kapovich, M. and Kleiner, B. Coarse Alexander duality 
and duality groups, 

J. Diff. Geom. 69 (2005), 279--352.

\bibitem{KR88} Kropholler, P. H. and Roller, M. A. Splittings of Poincar\'e 
duality groups,

Math. Zeit. 197 (1988), 421--438.

\bibitem{Th} Thurston, W. P. \textit{Three-Dimensional Geometry and Topology},

(edited by S. Levy), Princeton University Press, Princeton, N.J. (1997).

\end{thebibliography}
\end{document}